\documentclass[a4paper,10pt]{amsart}
\usepackage{amsmath,amsthm,amssymb,latexsym,enumerate,color,hyperref}
\usepackage{graphicx}

\usepackage{xcolor}

\usepackage[numbers]{natbib}

\numberwithin{equation}{section}

\begin{document}

\newtheorem{thm}{Theorem}[section]
\newtheorem{prop}[thm]{Proposition}
\newtheorem{lem}[thm]{Lemma}
\newtheorem{cor}[thm]{Corollary}
\newtheorem{rem}[thm]{Remark}
\newtheorem*{defn}{Definition}

\newcommand{\DD}{\mathbb{D}}
\newcommand{\NN}{\mathbb{N}}
\newcommand{\ZZ}{\mathbb{Z}}
\newcommand{\QQ}{\mathbb{Q}}
\newcommand{\RR}{\mathbb{R}}
\newcommand{\CC}{\mathbb{C}}
\renewcommand{\SS}{\mathbb{S}}

\renewcommand{\theequation}{\arabic{section}.\arabic{equation}}

\newcommand{\supp}{\mathop{\mathrm{supp}}}    

\newcommand{\re}{\mathop{\mathrm{Re}}}   
\newcommand{\im}{\mathop{\mathrm{Im}}}   
\newcommand{\dist}{\mathop{\mathrm{dist}}}  
\newcommand{\link}{\mathop{\circ\kern-.35em -}}
\newcommand{\spn}{\mathop{\mathrm{span}}}   
\newcommand{\ind}{\mathop{\mathrm{ind}}}   
\newcommand{\rank}{\mathop{\mathrm{rank}}}   
\newcommand{\Fix}{\mathop{\mathrm{Fix}}}   
\newcommand{\codim}{\mathop{\mathrm{codim}}}   
\newcommand{\conv}{\mathop{\mathrm{conv}}}   
\newcommand{\epsi}{\mbox{$\varepsilon$}}
\newcommand{\eps}{\mathchoice{\epsi}{\epsi}
{\mbox{\scriptsize\epsi}}{\mbox{\tiny\epsi}}}
\newcommand{\cl}{\overline}
\newcommand{\pa}{\partial}
\newcommand{\ve}{\varepsilon}
\newcommand{\zi}{\zeta}
\newcommand{\Si}{\Sigma}
\newcommand{\cA}{{\mathcal A}}
\newcommand{\cG}{{\mathcal G}}
\newcommand{\cH}{{\mathcal H}}
\newcommand{\cI}{{\mathcal I}}
\newcommand{\cJ}{{\mathcal J}}
\newcommand{\cK}{{\mathcal K}}
\newcommand{\cL}{{\mathcal L}}
\newcommand{\cN}{{\mathcal N}}
\newcommand{\cR}{{\mathcal R}}
\newcommand{\cS}{{\mathcal S}}
\newcommand{\cT}{{\mathcal T}}
\newcommand{\cU}{{\mathcal U}}
\newcommand{\OM}{\Omega}
\newcommand{\B}{\bullet}
\newcommand{\ol}{\overline}
\newcommand{\ul}{\underline}
\newcommand{\vp}{\varphi}
\newcommand{\AC}{\mathop{\mathrm{AC}}}   
\newcommand{\Lip}{\mathop{\mathrm{Lip}}}   
\newcommand{\es}{\mathop{\mathrm{esssup}}}   
\newcommand{\les}{\mathop{\mathrm{les}}}   
\newcommand{\nid}{\noindent}
\newcommand{\pzr}{\phi^0_R}
\newcommand{\pir}{\phi^\infty_R}
\newcommand{\psr}{\phi^*_R}
\newcommand{\pow}{\frac{N}{N-1}}
\newcommand{\ncl}{\mathop{\mathrm{nc-lim}}}   
\newcommand{\nvl}{\mathop{\mathrm{nv-lim}}}  
\newcommand{\la}{\lambda}
\newcommand{\La}{\Lambda}    
\newcommand{\de}{\delta}    
\newcommand{\fhi}{\varphi} 
\newcommand{\ga}{\gamma}    
\newcommand{\ka}{\kappa}   

\newcommand{\core}{\heartsuit}
\newcommand{\diam}{\mathrm{diam}}

\newcommand{\lan}{\langle}
\newcommand{\ran}{\rangle}
\newcommand{\tr}{\mathop{\mathrm{tr}}}
\newcommand{\diag}{\mathop{\mathrm{diag}}}
\newcommand{\dv}{\mathop{\mathrm{div}}}

\newcommand{\al}{\alpha}
\newcommand{\be}{\beta}
\newcommand{\Om}{\Omega}
\newcommand{\na}{\nabla}

\newcommand{\cC}{\mathcal{C}}
\newcommand{\cM}{\mathcal{M}}
\newcommand{\nr}{\Vert}
\newcommand{\De}{\Delta}
\newcommand{\cX}{\mathcal{X}}
\newcommand{\cP}{\mathcal{P}}
\newcommand{\om}{\omega}
\newcommand{\si}{\sigma}
\newcommand{\te}{\theta}
\newcommand{\Ga}{\Gamma}

\title[An interpolating inequality]{An interpolating inequality for solutions of uniformly elliptic equations}

\author{Rolando Magnanini} 
\address{Dipartimento di Matematica ed Informatica ``U.~Dini'',
Universit\` a di Firenze, viale Morgagni 67/A, 50134 Firenze, Italy.}
    \email{magnanini@unifi.it}
    \urladdr{http://web.math.unifi.it/users/magnanin}

\author{Giorgio Poggesi}
\address{Department of Mathematics and Statistics, The University of Western Australia, 35 Stirling Highway, Crawley, Perth, WA 6009, Australia}
    \email{giorgio.poggesi@uwa.edu.au}
    \urladdr{https://research-repository.uwa.edu.au/en/persons/giorgio-poggesi-2}

%
%


\begin{abstract}
We extend an inequality for harmonic functions, obtained in \cite{MP3, PogTesi}, to the case of solutions of uniformly elliptic equations in divergence form, with merely measurable coefficients. The inequality for harmonic functions turned out to be a crucial ingredient in the study of the stability of the radial symmetry for Alexandrov's Soap Bubble Theorem and Serrin's problem. The proof of our inequality is based on a mean value property for elliptic operators stated and proved in \cite{Ca} and \cite{BH}. 
\end{abstract}

%
%
%

\maketitle

\raggedbottom

\subsection*{Copyright Statement} This is a preprint of the following chapter: 
\\
R. Magnanini and G. Poggesi (2021) An Interpolating Inequality for Solutions of Uniformly Elliptic Equations. In: V. Ferone, T. Kawakami, P. Salani, F. Takahashi (eds) Geometric Properties for Parabolic and Elliptic PDE's. Springer INdAM Series, vol 47. Springer, Cham. Reproduced with the permission of Springer Nature Switzerland AG 2021. The final authenticated version is available at \url{https://doi.org/10.1007/978-3-030-73363-6_11}  

\section{Introduction}

Let $\Om$ be a bounded domain in $\RR^N$, $N\ge 2$, and denote its boundary by $\Ga$. The volume of $\Om$ and the $(N-1)$-dimensional Hausdorff measure of $\Ga$ will be denoted, indifferently, by $|\Om|$ and $|\Ga|$. Let $A(x)$ be an $N\times N$ symmetric matrix whose entries $a_{ij}(x)$, $i,j=1,\dots, N$, are measurable functions in $\Om$. We assume that $A(x)$ satisfies the (uniform) ellipticity condition:
\begin{equation}
\label{ellipticity}
\la\,|\xi|^2 \le \lan A(x)\,\xi, \xi\ran \le \La\,|\xi|^2 \ \mbox{ for any } \ x\in\Om, \ \xi\in\RR^N.
\end{equation}
Here, $\la$ and $\La$ are positive constants.
Associated to $A(x)$ we consider a uniformly elliptic linear operator $L$ in divergence form, defined formally by
\begin{equation}
\label{operator}
Lv=\dv [ A(x)\,\na v],
\end{equation}
for every $x\in\Om$.
\par
In what follows, we shall use two scaling invariant quantities: for $1\le p\le\infty$ the number $\nr v\nr_{p,\Om}$ will denote the $L^p$-norm of a measurable function $v:\Om\to\RR$ with respect to the normalized Lebesgue measure $dx/|\Om|$
and, for $0<\al\le 1$, we define the scaling invariant H\"older seminorm
\begin{equation}
\label{seminorm}
[v]_{\al,\Om}=\sup \left\{\left(\frac{d_\Om}{2}\right)^\al\frac{|v(x_1) - v(x_2)|}{|x_1 - x_2|^\al} \, : \,  x_1,x_2 \in \ol{\Om}, \, x_1 \neq x_2 \right\},
\end{equation}
where $d_\Om$ is the diameter of $\Om$. Also, the mean value of $v$ on $\Om$ will be indicated by $v_\Om$.  \par
For $0<\al\le 1$, we let $\Si_\al(\Om)$ be the set of weak solutions $v$ of class $C^{0,\al} (\ol{\Om})$ of $Lv=0$ in $\Om$. We denote by $B_r$ and $S_r$ the ball and sphere of radius $r$ centered at the origin. 
To avoid unessential technicalities, we state here our main result in the case in which $\Om$ is a ball. The case of general domains will be treated later on.
%
%

\begin{thm}
\label{thm:ball}
Take $p\in [1,\infty)$. There exists a positive constant $K$
%
%
such that,  for any $v\in\Si_\al(B_r)$, it holds that
\begin{equation}
\label{eq:ball-Lp-stability-generic-v-2}  
\max_{S_r} v - \min_{S_r} v \le K\, [v]_{\al,B_r}^{ \frac{N}{N+ \al \, p} } \, \nr v-v_{B_r}\nr_{p, B_r}^\frac{\al p}{N+ \al p}.
\end{equation} 
\par
Moreover, \eqref{eq:ball-Lp-stability-generic-v-2} is optimal in the sense that the equality sign holds for some $v\in\Si_\al(B_r)$. Finally, we have that
\begin{equation}
\label{eq:ball_costante_a_N}
K\le 2 \left( 1 +\frac{ \al p}{N} \right) \left( \frac{N}{\al p}\right)^{\frac{\al p}{N+ \al p}} \,\left( \frac{C}{c} \right)^\frac{\al N}{N+\al p},
\end{equation}
where $c, C$, with $c\le C$, are two constants that only depend on $N, \la$ and $\La$. 
\end{thm}
\par
We recall that, by De Giorgi-Nash-Moser's theorem, we have that a solution of $Lu=0$ is locally of class $C^{0,\al} (\Om)$ for some $\al\in (0,1]$ that depends on $N, \la$ and $\La$. Moreover
%
%
that regularity can be extended up to the boundary provided
%
%
$u$ is H\"older-continuous on $\Ga$
and $\Ga$ is sufficiently smooth -- e.g., $\Ga$ satisfies a uniform exterior cone condition (see \cite[Theorem 8.29]{GT}) or, more in general, condition (A) defined in \cite[pag. 6]{LU} (see \cite[Theorem 1.1 of Chapter 4]{LU}). 
\par

The reader's attention should be focused on the quantitative character of \eqref{eq:ball-Lp-stability-generic-v-2}. This says that the oscillation of a solution of an elliptic equation can be controlled, up to the boundary, by its $L^p$-norm in the domain, provided some a priori information is given on its
H\"older seminorm. 
\par
The effectiveness of an inequality like \eqref{eq:ball-Lp-stability-generic-v-2} can be understood from an important application of it, that was first given in \cite{MP}, and then refined in \cite{MP2, MP3, Pog, PogTesi} (see also \cite{Ma} for a survey on those issues). There, rougher versions of \eqref{eq:ball-Lp-stability-generic-v-2} for harmonic functions were used to obtain quantitative rigidity estimates for the spherical symmetry in two celebrated problems in differential geometry and potential theory: Alexandrov's Soap Bubble Theorem and Serrin's overdetermined problem (see \cite{Al1, Al1transl, Al2, Se} for the original rigidity results). 
%
%
Another application of an inequality like \eqref{eq:ball-Lp-stability-generic-v-2} can be found in \cite{DPV}.
\par
Theorem \ref{thm:ball} improves the result obtained in \cite[Lemma 3.14]{PogTesi} (and hence the previous ones) from various points of view. As already mentioned, it extends the analogous estimates obtained for harmonic functions to the case of a uniformly elliptic linear operator in divergence form with \textit{merely} measurable coefficients. Moreover, it removes the restriction of smallness of the term $\nr v-v_\Om \nr_{p, B}$ that was present in the previous inequalities. In doing so, it clears up which are the essential ingredients to consider to obtain a best possible bound. Finally, it also relaxes the former Lipschitz assumption on the solutions to a weaker H\"older continuous a priori information. 
\par
The proof of the existence of the constant $K$ in \eqref{eq:ball-Lp-stability-generic-v-2} is obtained by a quite standard variational argument. The necessary compactness of the optimizing sequence is derived from a rougher version of \eqref{eq:ball-Lp-stability-generic-v-2}, which it is proved in Lemma \ref{lem:ball}.
The proof of this lemma extends the arguments, first used in \cite{MP} and refined in \cite{MP2, MP3, PogTesi} for harmonic functions, to the case of an elliptic operator.
The crucial ingredient to do so is a mean value theorem for elliptic equations in divergence form (see Theorem \ref{thm:meanCaffarelli}) the proof of which is sketched in \cite[Remark at page 9]{Ca} and given with full details in \cite[Theorem 6.3]{BH}.
\par
The proof of Theorem \ref{thm:ball} is given in Section \ref{sec:ball}. There, we also provide a proof for the case of smooth domains. In this case, the constant $K$ also
depends on the ratio between the diameter and the radius
of a uniform interior touching ball for the relevant domain. 
In Section \ref{sec:cone-john}, we show that the proof's scheme can be extended to two instances of non-smooth domains: those satisfying either the uniform interior cone condition or the so-called local John's condition. The dependence of $K$ on the relevant parameters follows accordingly.

\section{The inequality in a ball and in smooth domains}
\label{sec:ball}

We recall the already mentioned result introduced by L. Caffarelli \cite[Remark on page 9]{Ca}, the proof of which is provided in full details in \cite[Theorem 6.3]{BH}.
In what follows, $B_r (x_0)$ denotes the ball of radius $r$ centered at $x_0$.

%
%

\begin{thm}[Mean Value Property for Elliptic Operators]
\label{thm:meanCaffarelli}
Let $\Om$ be an open subset of $\RR^N$. Let $L$ be the elliptic operator defined by \eqref{ellipticity}-\eqref{operator} and pick any $x_0 \in \Om$.
Then, there exist two constants $c$, $C$ that only depend on $N, \la$ and $\La$, and,
for
$0<r< r_0$ with $r_0 \ge \dist(x_0,\Ga)/C$,
%
%
an increasing family of domains $D_r(x_0)$ which satisfy the properties:
\begin{enumerate}[(i)]
\item $B_{c r} (x_0) \subset D_r(x_0) \subset B_{C r} (x_0)$;
%
%
\item  for any $v$ satisfying $Lv \ge 0$, we have that
\begin{equation}
\label{eq:meanvalueprop}
v(x_0) \le \frac{1}{| D_r(x_0) |} \int_{D_r(x_0)} v(y) \, dy \le \frac{1}{| D_\rho (x_0) |} \int_{D_\rho(x_0)} v(y) \, dy,
\end{equation}
for any $0<r<\rho< r_0$.
\end{enumerate}
\end{thm}
Issues  related to this theorem and the study of the geometric properties of the sets $D_r(x_0)$ 
%
%
have been recently studied by I. Blank and his collaborators in \cite{ArmB, AryB, BBL}.

%
%
%
%

\subsection{The inequality for a ball}
 We begin our presentation by considering the
case of a ball. This will avoid extra technicalities. We will later show how to
extend our arguments to other types of domains.
\par
The  following lemma gives a rough estimate for sub-solutions of the elliptic equation $Lv=0$.

\begin{lem}
\label{lem:ball}
Take $p\ge 1$. Let $v\in C^{0,\al} (\ol{B_r})$, $0<\al\le 1$, be a weak solution of $Lv \ge 0$ in $B_r$.
Then we have that
\begin{equation}
\label{eq:ball-Lp-stability-generic-v}  
\max_{S_r} v - \min_{S_r} v \le 
2 \left( 1 +\frac{ \al p}{N} \right) \left( \frac{N}{\al p}\right)^{\frac{\al p}{N+ \al p}}
\,\left( \frac{C}{c} \right)^\frac{\al N}{N+\al p}  [v]_{\al,B_r}^{ \frac{N}{N+ \al \, p} } \, \nr v-v_{B_r} \nr_{p, B_r}^\frac{\al p}{N+ \al p}.
\end{equation} 
\end{lem}

%
%

\begin{proof}
Without loss of generality, we can assume that $v_{B_r}=0$. Let $x_1$ and $x_2$ be points on $S_r$ that respectively minimize and maximize $v$ on $S_r$ and, for 
$0<\si < r$,
define the two points 
$y_j=x_j-\si x_j/r$, $j=1, 2$.
Notice that $x_j/r$ is the exterior unit normal vector to $S_r$ at the point $x_j$.
\par

By \eqref{seminorm} and the fact that $2 r$ is the diameter of $B_r$, we have that
\begin{equation}
\label{eq:ball_prova_seminorm_genericv}
| v(x_j) | \le |v(y_j)| + [v]_{\al, B_r} \left(\frac{\si}{r}\right)^\al , \ j=1,2.
\end{equation}
%
%
%
Being $0 < \si < r$, we have that $B_{\si }(y_j ) \subset \Om $.
%
%
Thus, we apply Theorem \ref{thm:meanCaffarelli} by choosing $x_0=y_j$, $j=1,2$, and $r=\si/C$. By item (i), we have that 
\begin{equation}\label{eq:ball_inclusions}
B_{\frac{c}{C} \si } (y_j) \subset D_{ \frac{\si}{C}}(y_j ) \subset B_\si  (y_j) \subset B_r , \quad j=1,2.
\end{equation}
Also, item (ii) gives that 
\begin{multline}
\label{eq:ball_mean+Holder}
|v(y_j) | \le \frac{1}{|D_\frac{\si}{C} (y_j )|} \int_{D_\frac{\si}{C} (y_j)} | v | \, dy \le
\\
\frac{1}{ |D_\frac{\si}{C} (y_j)|^{1/p}}  \, \left[\int_{ D_\frac{\si}{C} (y_j) }  |v  |^p\,dy\right]^{1/p} \le 
|B|^{-\frac1{p}}\left(\frac{C}{c\,\si} \right)^{N/p} \, \left(\int_{B_r}|v |^p\,dy\right)^{1/p} .
\end{multline}
The second inequality is a straightforward application of H\"older's inequality and, in the last inequality, we used \eqref{eq:ball_inclusions}, that also gives that
$$
|D_\frac{\si}{C} (y_j )| \ge |B| \left( \frac{c}{C} \right)^N \si^N.
$$

Putting together \eqref{eq:ball_prova_seminorm_genericv} and \eqref{eq:ball_mean+Holder} yields that
\begin{equation}
\label{eq:ball_stepprimadellasceltadisigma}
\max_{S_r} v - \min_{S_r} v \le 2 \, \left[\left( \frac{C}{c} \right)^{N/p}\nr v \nr_{p, B_r} \, \left(\frac{\si}{r}\right)^{-N/p} + [v]_{\al, B_r}\left(\frac{\si}{r}\right)^\al\right],
\end{equation}
for every $0<\si <r$.
\par
Therefore, by minimizing the right-hand side of the last inequality, we can conveniently choose 
\begin{equation}
\label{eq:ball_choicesigma}
\frac{\si^*}{r}=\left[   \frac{N}{\al p}\,\left( \frac{C}{c}\right)^{N/p}  \frac{\nr v \nr_{p, B_r} }{[v]_{\al, B_r} }\right]^{p/(N+ \al p)} 
\end{equation}
and obtain \eqref{eq:ball-Lp-stability-generic-v} if $\si^* < r$.
\par
On the other hand, if $\si^* \ge r$, by \eqref{seminorm} 
we can write:
$$
\max_{S_r} v - \min_{S_r} v \le 2^\al \,[v]_{\al,B_r}  \le 2^\al \,[v]_{\al,B_r} \left(\frac{\si^*}{r}\right)^\al.
$$
Thus, \eqref{eq:ball_choicesigma} gives
$$
\max_{S_r} v - \min_{S_r} v \le 2^\al \left(\frac{N}{\al p}\right)^\frac{ \al p}{N+\al p} \left( \frac{C}{c} \right)^\frac{ \al N}{N+\al p} [v]_{\al,B_r}^{ \frac{N}{N+ \al \, p} } \, \nr v \nr_{p, B_r}^\frac{\al p}{N+ \al p}.
$$
Therefore, \eqref{eq:ball-Lp-stability-generic-v} always holds true, being $2^\al\le 2(1+\al p/N)$.
\end{proof}

\medskip

\begin{proof}[\bf Proof of Theorem \ref{thm:ball}]
Lemma \ref{lem:ball} tells us that \eqref{eq:ball-Lp-stability-generic-v-2} and \eqref{eq:ball_costante_a_N} hold with 
$$
K=\sup\Bigl\{ \max_{S_r} v - \min_{S_r} v : v\in \Si_\al(B_r) \mbox{ with } [v]_{\al,B_r}^{ \frac{N}{N+ \al \, p} } \, \nr v-v_{B_r} \nr_{p, B_r}^\frac{\al p}{N+ \al p}\le 1\Bigr\}.
$$
\par
We are thus left to prove the existence of a $v\in\Si_\al(B_r)$ that attains the supremum. Again, we assume that $v_{B_r}=0$ in the supremum and take a maximizing sequence of functions $v_n$, that is
$$
[v_n]_{\al,B_r}^{ \frac{N}{N+ \al \, p} } \, \nr v_n\nr_{p, B_r}^\frac{\al p}{N+ \al p}\le 1
 \ \mbox{ and } \ \max_{S_r} v_n - \min_{S_r} v_n \to K \ \mbox{ as } \ n\to\infty.
$$ 
Observe that
$$
\nr v_n\nr_{p, B_r}\le 2^\frac{\al N}{N+\al p}, \ n\in\NN,
$$
since
$$
\nr v \nr_{p, B_r}=\nr v - v_{B_r}\nr_{p, B_r}\le 2^\al [v ]_{\al,B_r}, \ v\in\Si_\al(B_r).
$$
\par
We can then extract a subsequence of functions, that we will still denote by $v_n$, that weakly converge in $L^p(B_r)$ to a function $v\in L^p(B_r)$. By the mean value property of
Theorem \ref{thm:meanCaffarelli}, the sequence converges  uniformly to $v$ on the compact subsets of $B_r$, and hence $v$ satisfies the mean value property of Theorem \ref{thm:meanCaffarelli} in $B_r$.
%
The converse of the mean value theorem (see, e.g., \cite[Theorem 1.2]{ArmB})
then gives that $Lv=0$ in $B_r$.
\par
Next, we fix $x_1, x_2\in B_r$ with $x_1\ne x_2$. Since
$$
r^\al\,\frac{|v_n(x_1)-v_n(x_2)|}{|x_1-x_2|^\al}\le [v_n]_{\al, B_r}\le \nr v_n\nr_{p,B_r}^{- \frac{\al p}{N} },
$$
the local uniform convergence and the semicontinuity of the $L^p$-norm with respect to weak convergence give that
$$
r^\al\frac{|v(x_1)-v(x_2)|}{|x_1-x_2|^\al}\le \nr v\nr_{p,B_r}^{ - \frac{\al p}{N} }.
$$
Since $x_1$ and $x_2$ are arbitrary, we infer that $ [v]_{\al, B_r}\,\nr v\nr_{p,B_r}^{ \frac{\al p}{N} }  \le 1$. This means that $v$ extends to a function of class $C^{0,\al}(\ol{B_r})$. 
\par
If we now prove that $v_n\to v$ uniformly on $S_r$, we will have that
$$
K=\lim_{n\to\infty} \left(\max_{S_r} v_n - \min_{S_r} v_n\right)=\max_{S_r} v - \min_{S_r} v,
$$
and the proof would be complete.
For any $x\in S_r$ and $y\in B_r$, we can easily show that
\begin{multline*}
\limsup_{n\to\infty}|v_n(x)-v(x)|\le  r^{-\al}\,|x-y|^\al \limsup_{n\to\infty} [v_n]_{\al, B_r}+|v(y)-v(x)|\le \\
r^{-\al}\,|x-y|^\al \nr v\nr_{p,B_r}^{ - \frac{\al p}{N} }+|v(y)-v(x)|.
\end{multline*}
Since $y\in B_r$ is arbitrary and $v$ is continuous up to $S_r$, the right-hand side can be made arbitrarily small, and hence we infer that $v_n$ converges to $v$ pointwise on $S_r$. The convergence turns out to be uniform on $S_r$. In fact, if $x_n\in S_r$ maximizes $|v_n-v|$ on $S_r$ then by compactness $x_n\to x$ as $n\to\infty$ for some $x\in S_r$, modulo a subsequence. Thus,
\begin{multline*}
\max_{S_r}|v_n-v|=|v_n(x_n)-v(x_n)|\le  \\
r^{-\al}\,|x_n-x|^\al  [v_n]_{\al, B_r}+|v_n(x)-v(x)|+|v(x)-v(x_n)|,
\end{multline*}
and the right-hand side vanishes as $n\to\infty$, by the continuity of $v$ and the pointwise convergence of $v_n$. The proof is complete.
\end{proof}

\subsection{The inequality for smooth domains}
The extension of Theorem \ref{thm:ball} to the case of bounded domains with boundary $\Ga$ of class $C^2$ is not difficult. We recall that such domains satisfy a \textit{uniform interior sphere condition}. In other words,
there exists $r_i>0$ such that for each $z \in \Ga$ there is a ball of radius $r_i$ contained in $\Om$  the closure of which intersects $\Ga$ only at $z$. 

\begin{thm}
\label{thm:C^2 domains}
Take $p\in [1,\infty)$. Let $\Om\subset \RR^N$ be a bounded domain with boundary $\Ga$ of class $C^2$ and let $L$ be the elliptic operator defined by \eqref{ellipticity}-\eqref{operator}. 
\par
If $v\in\Si_\al(\Om)$, then 
\begin{equation}
\label{eq:smooth-Lp-stability-generic-v}  
\max_{\Ga} v - \min_{\Ga} v \le K\, [v]_{\al,\Om}^{ \frac{N}{N+ \al \, p} } \, \nr v-v_{\Om}\nr_{p, \Om}^\frac{\al  p}{N+ \al p}
\end{equation}
for some optimal constant $K$.
Moreover, it holds that
\begin{equation}\label{eq:smooth_explicitboundK}
K\le \max\left[  2 \left(  1 +\frac{ \al p}{N} \right)  ,\left(\frac{d_\Om}{r_i}\right)^\al\right]\left( \frac{N}{\al p}\right)^{\frac{\al p}{N+ \al p}}\,\left( \frac{C}{c} \right)^\frac{\al N}{N+\al p}.
\end{equation}
\end{thm}

\begin{proof}
The proof runs similarly to that of Theorem \ref{thm:ball}. We just have to make some necessary changes to the proof of Lemma \ref{lem:ball}, 
\par
We take $x_1$ and $x_2$ in $\Ga$ that respectively minimize and maximize $v$ on $\Ga$ and define the corresponding $y_1, y_2$ by
$y_j=x_j-\si \nu (x_j)$, $j=1, 2$,
where $\nu(x_j)$ is the exterior unit normal vector to $\Ga$ at the point $x_j$. This time we use the restriction $0<\si < r_i$, so that $B_\si(y_j)\subset\Om$, $j=1, 2$.
\par
Next, we must replace \eqref{eq:ball_prova_seminorm_genericv} by
\begin{equation}\label{eq:smooth-disuguaglianzaconseminorma}
| v(x_j) | \le |v(y_j)| + [v]_{\al, \Om} \left(\frac{2\si}{d_\Om}\right)^\al , \ j=1,2,
\end{equation}
and
\eqref{eq:ball_mean+Holder} by
$$
|v(y_j) | \le 
|B|^{-\frac1{p}}\left(\frac{C}{c\,\si} \right)^{N/p} \, \left(\int_\Om|v |^p\,dy\right)^{1/p} , \ j=1,2.
$$
Thus, we arrive at 
\begin{equation}\label{eq:smooth_provoarichiamarla}
\max_{ \Ga } v - \min_{ \Ga } v \le 2 \, \left[\left( \frac{C}{c} \right)^{N/p}\nr v \nr_{p, \Om} \, \left(\frac{2\si}{d_\Om}\right)^{-N/p} + [v]_{\al, \Om}\left(\frac{2\si}{d_\Om}\right)^\al\right]
\end{equation}
for $0<\si< r_i$, in place of \eqref{eq:ball_stepprimadellasceltadisigma}. Here, we used that $|\Om|\le |B|\, (d_\Om/2)^N$.
%
%

By minimizing the right-hand side of \eqref{eq:smooth_provoarichiamarla}, this time we can choose
$$
\frac{2\si^*}{d_\Om}=\left[  \frac{N}{\al p}\, \left(\frac{C}{c}\right)^{N/p}  \frac{\nr v \nr_{p, \Om} }{[v]_{\al,\Om} }\right]^{p/(N+ \al p)} ,
$$
%
%
and obtain \eqref{eq:smooth-Lp-stability-generic-v} and \eqref{eq:smooth_explicitboundK} if $\si^*<r_i$.

On the other hand, if $\si^* \ge r_i$, \eqref{seminorm} gives:
\begin{multline*}
\max_{\Ga} v - \min_{\Ga} v  \le 2^\al\,[v]_{\al,\Om}\le \left( \frac{2\si^*}{r_i}\right)^\al [v]_{\al,\Om}=\\
\left( \frac{N}{\al p} \right)^{ \frac{\al p}{N+\al p} } \left( \frac{C}{c}\right)^\frac{\al N}{N+\al p}
\left( \frac{d_\Om}{r_i}\right)^\al [v]_{\al,\Om}^{ \frac{N}{N+ \al \, p} } \, \nr v \nr_{p, \Om}^\frac{\al  p}{N+ \al p}.
\end{multline*}
Again, \eqref{eq:smooth-Lp-stability-generic-v} and \eqref{eq:smooth_explicitboundK} hold true.
\end{proof}

\begin{rem}
{\rm
Theorem \ref{thm:C^2 domains} can be compared with \cite[Lemma 3.14]{PogTesi},
that was proved for the Laplace operator. In that case, we have that $c=C=1$ and the seminorm in \eqref{seminorm} can be replaced by the maximum of $(d_\Om/2)\,|\na v|$ on $\Ga$, provided $\Ga$ is sufficiently smooth.}
\end{rem}

\section{The inequality for two classes of non-smooth domains}
\label{sec:cone-john}

In this section, for future reference, we consider and carry out some details for two cases of domains with non-smooth boundary. 

\subsection{Domains with corners}
Given $\te\in [0, \pi /2]$ and $h>0$, we say that $\Om$ satisfies the \textit{$(\te, h)$-uniform interior cone condition}, if for every $x \in \Ga$ there exists a finite right spherical cone $\cC_x$ (with vertex at $x$ and axis in some direction $e_x$), having opening width $\te$ and height $h$, such that 
$$
\cC_x \subset \ol{\Om} \quad \text{and} \quad \ol{\cC}_x \cap \Ga = \left\lbrace x \right\rbrace.
$$

\begin{thm}
\label{thm:cone_Lp-estimate-oscillation-generic-v}
Take $p\in [1,\infty)$. Let $\Om\subset \RR^N$ be a bounded domain satisfying the $(\theta, h)$-uniform interior cone condition and let $L$ be the elliptic operator defined by \eqref{ellipticity}-\eqref{operator}. 
\par
If $v\in\Si_\al(\Om)$, then
%
%
%
\eqref{eq:smooth-Lp-stability-generic-v} holds true
for some optimal constant $K$.
Moreover, we have that
\begin{equation}\label{eq:cone_explicitboundK}
K \le \max\left[2 \left( 1 +\frac{ \al p}{N} \right) , \left(\frac{d_\Om}{h}\right)^\al \left( 1+\sin{ \theta } \right)^\al \right]\left( \frac{N}{\al p}\right)^{\frac{\al p}{N+ \al p}}\,\left( \frac{C}{c \, \sin{ \theta } } \right)^\frac{\al N}{N+\al p}.
\end{equation}
\end{thm}

\begin{proof}
The proof runs similarly to that of Theorem \ref{thm:C^2 domains}. We just have to take care of the bound for $K$.
\par
Let $x_1$ and $x_2$ be the usual extremum points for $v$ on $\Ga$. This time, instead, we define the two points $y_j=x_j-\si e_{x_j}$, $j=1, 2$, for
$$
0<\si < \frac{h}{1+ \sin\te}.
$$
Notice that, in view of the $(\theta, h)$-uniform interior cone condition, the ball $B_{\si\sin{\theta} }(y_j)$ is contained in $\Om$. 
Thus, by proceeding as in the proof of Theorem \ref{thm:C^2 domains} (this time applying Theorem \ref{thm:meanCaffarelli} with $r= \frac{\sin{\theta} }{C} \si$ and $x_0 = y_j$, $j=1,2$), we arrive at the inequality
\begin{equation*}
\max_{\Ga} v - \min_{\Ga} v \le 2 \, \left[ \left( \frac{ C   }{ c \, \sin{\theta} } \right)^{N/p} \nr v \nr_{p, \Om} \left(\frac{2\si}{d_\Om}\right)^{-N/p} + [v]_{\al, \Om}\left(\frac{2\si}{d_\Om}\right)^\al \right] ,
\end{equation*}
for every $0<\si < h/(1+ \sin\te)$. Hence, this time we can choose
$$
\frac{2\si^*}{d_\Om}= \left[  \frac{N}{\al p}\, \left(\frac{C}{c \, \sin{\theta} }\right)^{N/p}  \frac{\nr v \nr_{p, \Om} }{[v]_{\al,\Om} }\right]^{p/(N+ \al p)} ,
$$
and obtain \eqref{eq:smooth-Lp-stability-generic-v} and \eqref{eq:cone_explicitboundK} if $\si^* < h/(1+ \sin\te)$.
\par
On the other hand, if $\si^* \ge h/(1+ \sin\te)$, by \eqref{seminorm} we have that
\begin{multline*}
\max_{\Ga} v - \min_{\Ga} v  \le 2^\al\,[v]_{\al,\Om} \le \left( \frac{2 \si^*}{h}\right)^\al \left( 1+\sin{\theta} \right)^\al [v]_{\al,\Om}= \\
\left( \frac{N}{\al p} \right)^{ \frac{\al p}{N+\al p} } \left( \frac{C}{c \, \sin{\theta} }\right)^\frac{\al N}{N+\al p}
\left( \frac{d_\Om}{h}\right)^\al \left( 1+ \sin{\theta} \right)^\al [v]_{\al,\Om}^{ \frac{N}{N+ \al \, p} } \, \nr v \nr_{p, \Om}^\frac{\al  p}{N+ \al p}.
\end{multline*}
Again, \eqref{eq:smooth-Lp-stability-generic-v} and \eqref{eq:cone_explicitboundK} hold true.
\end{proof}

\subsection{Locally John's domains}
\label{sec:John}
Following \cite[Definition 3.1.12]{HMT}, we say that a bounded domain $\Om \subset \RR^N$ satisfies the {\it $(b_0,R)$-local John condition} if there exist two constants, $b_0 > 1$ and $R>0$, with the following properties. For every $x \in \Ga$ and $r \in \left( 0,R \right]$ we can find $x_r \in B_r (x) \cap \Om$ such that 
$B_{r/b_0}(x_r) \subset \Om$.
Also, for each $z$ in the set $\De_r (x)$ defined by $B_r(x)\cap\Ga$, we can find a rectifiable path 
$\ga_z : [0, 1] \to \ol{\Om}$, with length $ \le b_0 r$, such that
$\ga_z (0) = z$, $\ga_z (1) = x_r$, and
\begin{equation}\label{eq:localJohncondition}
\dist(\ga_z (t), \Ga ) >\frac{|\ga_z (t) - z|}{b_0} \ \text{ for any } \ t>0.
\end{equation}
The constants $b_0, R$, the point $x_r$, and the curve $\ga_z$ are respectively called \textit{John's constants, John's center (of $\De_r(x)$), and John's path}.
The class of domains satisfying the local John condition is huge and contains, among others, the so-called \textit{non-tangentially accessible} domains (see \cite[Lemma 3.1.13]{HMT}).

\begin{thm}
\label{thm:John-Lp-estimate-oscillation-generic-v}
Take $p\in [1,\infty)$. Let $\Om\subset \RR^N$ be a bounded domain satisfying the $(b_0,R)$-local John condition and let $L$ be the elliptic operator defined by \eqref{ellipticity}-\eqref{operator}. 
\par
If $v\in\Si_\al(\Om)$, then 
\eqref{eq:smooth-Lp-stability-generic-v} holds true for some optimal constant $K$.
Moreover, we have that
\begin{equation}
\label{eq:John_explicitboundK}
K \le \max\left[2 \left( 1 +\frac{ \al p}{N} \right) , \left(\frac{d_\Om b_0}{R}\right)^\al \right] \left( \frac{N}{\al p}\right)^{\frac{\al p}{N+ \al p}}
\,\left( \frac{C b_0 }{c } \right)^\frac{\al N}{N+\al p}.
\end{equation}
\end{thm}

\begin{proof}
Let $x$ be one of the usual extremum points for $v$ on $\Ga$. Let $\ga_x$ be a John's path from $x$ to the John's center $x_R$ of $\De_R(x)$. 
Since $B_{ R /b_0}(x_R) \subset \Om$ we have that
$$ 
| x - x_R | \ge \dist(x_R,\Ga)>\frac{R}{b_0}.
$$  
Thus, for $0<\si < R/b_0$,
we can find a point $y$ on the John's curve $\ga_x$ such that 
$ |x-y| = \si$.
Hence, by \eqref{seminorm} we have that \eqref{eq:smooth-disuguaglianzaconseminorma} still holds true.

In view of \eqref{eq:localJohncondition} we have that $B_{\si / b_0}(y)\subset\Om$. 
Thus, as done to obtain \eqref{eq:ball_mean+Holder} (this time applying Theorem \ref{thm:meanCaffarelli} with $r= \si / (C b_0)$ and $x_0 = y$),
we get that
$$
|v(y) | \le \left(\frac{|\Om|}{|B|}\right)^{\frac{1}{p}}  \left[ \frac{C \, b_0}{ c \, \si  } \right]^{N/p} \, \nr v\nr_{p,\Om}.
$$

This, \eqref{eq:smooth-disuguaglianzaconseminorma}, and the inequality $| \Om | \le |B| \left( d_\Om /2  \right)^N$ then yield that
\begin{equation*}
\max_{\Ga} v - \min_{\Ga} v \le 2 \, \left[ \left( \frac{ C \, b_0   }{ c } \right)^{N/p} \left( \frac{2 \, \si}{d_\Om} \right)^{- N/p} \nr v \nr_{p, \Om} + [v]_{\al,\Om} \,
 \left( \frac{2 \, \si}{d_\Om} \right)^\al   \right] ,
\end{equation*}
for every $0<\si < R/b_0$.
Hence, this time we can choose 
\begin{equation}\label{eq:John_choicesigma}
\frac{2 \, \si^* }{d_\Om} = \left[ \frac{N}{\al \, p }  \left( \frac{ C \, b_0   }{ c } \right)^{N/p} \frac{ \nr v \nr_{p, \Om} }{ [v]_{\al , \Om}   } \right]^{p/(N+ \al p)}  ,
\end{equation}
and have that \eqref{eq:smooth-Lp-stability-generic-v} and \eqref{eq:John_explicitboundK} hold true if 
$\si^* < R/b_0$. 
\par
On the other hand if $\si^* \ge R/b_0$,
since by \eqref{seminorm} it holds that
$$ \max_{\Ga} v - \min_{\Ga} v  = v(x_1) - v(x_2) \le 2^\al \, [v]_{\al,\Om} 
\left( \frac{\si^* b_0 }{R}\right)^\al \le  [v]_{\al,\Om} \left(2 \, \si^*\right)^\al \left( \frac{b_0}{R} \right)^\al ,$$
by \eqref{eq:John_choicesigma} we immediately get
\begin{equation*}
\max_{\Ga} v - \min_{\Ga} v \le \left( \frac{d_\Om b_0}{R} \right)^\al \left( \frac{N}{\al p}\right)^{\frac{\al p}{N+ \al p}}\,
\left( \frac{C b_0 }{c } \right)^\frac{\al N}{N+\al p} \,  [v]_{\al,\Om}^{ \frac{N}{N+ \al \, p} } 
\, \nr v \nr_{p, \Om}^{ \al p/(N+ \al p)} .
\end{equation*}
Hence, \eqref{eq:smooth-Lp-stability-generic-v} and \eqref{eq:John_explicitboundK} still hold true.
\end{proof}

\section*{Acknowledgements}
The authors wish to thank Ivan Blank for useful discussions.
%
The authors are partially supported by the Gruppo Nazionale di Analisi Matematica, Probabilità e Applicazioni (GNAMPA) of the Istituto Nazionale di Alta Matematica (INdAM). The second author is member of AustMS and is supported by the Australian Laureate Fellowship FL190100081 "Minimal surfaces, free boundaries and partial differential equations" and the Australian Research Council Discovery Project DP170104880 "N.E.W. Nonlocal Equations at Work".

\bibliographystyle{plain}

\bibliography{References}

\end{document}